\documentclass[12pt]{iopart}

\usepackage{iopams}  
\usepackage{stmaryrd} 
\usepackage{amsthm}

\newtheorem{theorem}{Theorem}
\newtheorem{lemma}[theorem]{Lemma}
\newtheorem{proposition}[theorem]{Proposition}
\newtheorem{corollary}[theorem]{Corollary}
\newtheorem{assumption}[theorem]{Assumption}
\newtheorem{remark}[theorem]{Remark}

\newcommand{\R}{\mathbb{R}}
\newcommand{\ep}{\varepsilon}
\newcommand{\A}{\mathcal{A}}
\newcommand{\X}{\mathcal{X}}
\newcommand{\Y}{\mathcal{Y}}
\newcommand{\KK}{\mathbf{K}}
\newcommand{\ZZ}{\mathcal{Z}}
\newcommand{\Co}{\mathcal{C}}

\begin{document}

\title[A reaction-diffusion model of plankton communities]{Well-posedness, positivity and time asymptotics properties for a reaction-diffusion
model of plankton communities, involving a rational nonlinearity with singularity}

\author{Antoine Perasso$^{1}$, Quentin Richard$^{2}$, Irene Azzali$^{3}$ and Ezio Venturino$^{4}$}


\address{$^{1}$ UMR 6249 Chrono-Environnement, Universit\'e Bourgogne Franche-Comt\'e, 25000 Besan\c con, France. \\
$^{2}$ MIVEGEC, IRD, UMR CNRS 5290, Universit\'e de Montpellier, 34394 Montpellier, France. \\
$^{3}$ Department of Veterinary Sciences largo Paolo Braccini 2, Universit\`a di Torino, 10095 Grugliasco, Italy. \\
$^{4}$ Dipartimento di Matematica ``Giuseppe Peano'', Universit\`a di Torino, 10123 Torino, Italy, Member of the INdAM research group GNCS.}
\ead{quentin.richard@math.cnrs.fr (corresponding author)}
\vspace{10pt}
\begin{indented}
\item[]April 2020
\end{indented}

\begin{abstract}
In this work, we consider a reaction-diffusion system, modeling the interaction between nutrients, phytoplanktons and zooplanktons. Using a semigroup approach in $L^2$, we prove global existence, uniqueness and positivity of the solutions. The Holling type $2$ nonlinearities, i.e of rational type with singularity, are handled by providing estimates in $L^\infty$. The article finally exhibits some time asymptotics properties of the solutions. 
\vspace{2pc}

\end{abstract}


\submitto{Studies in Applied Mathematics}

%
%
%
%

\section{Introduction}

Red, or brown, tides are outbreaks of algae in the oceans, quite often harmful, that threaten aquatic life
and constitute a serious problem for the fishing industry and tourism. Models for plankton dynamics have
been devised since more than two decades, because the phytoplankton-zooplankton trophic interactions are at
the base of the food chain on our planet, \cite{edbrid} and disturbances to this basic ecosystem such as those
mentioned above may have
serious consequences that are far beyond the nutrition chain and may involve the worldwide oxygen production, see
\cite{Yadigar} and the references quoted there.
This is mainly the consequence of the unregulated human activities, \cite{and},
\textit{e.g.} utilization of chemical pesticides in agriculture
and the release of untreated wastewaters, \cite{D2014},
that ultimately flow into the shallow waters near the coastlines and
thus contribute to the raise in the organic nutrients concentration in the ocean, \cite{fri}.
The harmful algal blooms
deplete the water from its oxygen content and thereby threaten the life of aquatic creatures.
For these reasons it is important to be able to predict them and mathematical models are fundamental tool to achieve that goal, \cite{Franks-1997,SH1,SH2}.
Most of these models have been formulated by explicitly avoiding space, assuming that
the ocean environment properties are independent of time or position in space, \cite{ABM1}. 
But this is unrealistic as hydrodynamics plays an important role in the shaping of an aquatic
community, as well as factors as temperature, salinity, turbulent mixing intensity.
A consequence is the fact that spatial structures become possible in this context, both 
induced by the heterogeneities in the aquatic medium and by the trophic interactions, \cite{S1,S2}.
Thus multi-habitat and multi-patch formation is possible, \cite{ABM2}.



Large amplitude oscillations of plankton populations are predicted by theoretical analyses,
when nutrients abound in the ocean, \cite{ros,Gilpin,edbrid}, but are not confirmed by empirical data,
\cite{Slaughter,Goericke,Benndorf},
originating thus the paradox of enrichment, \cite{ros,Gilpin}.
The original Rosenzweig model has been modified to improve it, in particular accounting for the 
zooplankton vertical movement, following phytoplankton for feeding, \cite{FMAN}. The latter 
indeed distributes inhomogeneously in view of the diminishing light in the water with depth, due to absorption in the upper layers, \cite{Ray}.
The properties of the combined above mechanisms leading oscillations to settle to a stable coexistence equilibrium have been elucidated
in \cite{mor1,FMAN}. 

Based in part on these results, further explorations have been carried out in \cite{AMV},
including a depth-dependent vertical turbulent diffusion, providing a more realistic scenario. In this paper, we study the following model equations, where $t>0$ is time, $h\in [0,H]$ is the depth and others parameters are given in Table \ref{Table:param} (see \cite{AMV} for details). Moreover $p$, $n$ and $z$ respectively represent the phytoplankton and nutrients densities and the average density of zooplankton:
\begin{eqnarray}\label{sistema1}
\left\{
\begin{array}{rcl}
\frac{\partial n}{\partial t}&=& D\frac{\partial^2 n}{\partial h^2} -L_h(p) \left(\frac{n}{1+\chi n}\right), \\
\frac{\partial p}{\partial t}&=& D\frac{\partial^2 p}{\partial h^2 } + L_h(p) \left(\frac{n}{1+\chi n}\right)-zg(p)-m_pp, \\
\frac{dz}{dt}&=&\frac{kz(t)}{H} \displaystyle \int_0^H g(p)(t,h) \,dh -mz(t).
\end{array}
\right.
\end{eqnarray}
Due to the functional response $g$, \textit{i.e.} the predator ingestion rate of the zooplankton as a function of phytoplankton density, the latter model is a generalization of the one proposed in \cite{AMV}. The operator $L_h$ is given either by
\begin{equation}\label{Lh1}
L_h(p) = r \exp(-\gamma h)p
\end{equation}
or by
\begin{equation}\label{Lh2}
L_h(p) = r\exp\left( -\nu \int_0^{h}p(t,x)\,dx\right) p
\end{equation}
assuming, as the case may be, an exponential decay of light with increasing depth or a light attenuation due to phytoplankton self-shading. The parameters $r, \gamma$ and $\nu$ introduced in the the latter equations are also defined in Table \ref{Table:param}.
\begin{table}\label{Table:param}
\begin{center}
{\footnotesize \begin{tabular}{|c|c|}
\hline 
Notations & Definitions  \\ 
\hline 
D & Vertical turbulent diffusion \\ 
\hline 
H & Depth of water column \\ 
\hline 
$\chi$ & Inverse half-saturation density of nutrient intake \\ 
\hline 
$m, m_p$ & Zooplankton and phytoplankton mortality rates \\ 
\hline 
$k$ & Food utilization coefficient \\ 
\hline 
$r/\chi$ & Maximum phytoplankton growth rate  \\ 
\hline 
$\gamma$ & Light attenuation coefficient \\ 
\hline 
$\nu$ & Self-shading coefficient \\
\hline 
\end{tabular}} \caption{Parameters involved in the model}
\end{center}
\end{table}
Moreover, system (\ref{sistema1}) is equipped with the following boundary conditions:
\begin{eqnarray*}
&\frac{\partial n}{\partial h} (t,0) = 0 , \qquad n(t,H)=n_H, \\
&\frac{\partial p}{\partial h} (t,0) =0, \qquad \frac{\partial p}{\partial h} (t,H) = 0,
\end{eqnarray*}
for every $t>0$, where $n_H\geq 0$ is constant. We also add some initial conditions:
\begin{equation*}
n(0,h) = n_0(h), \qquad p(0,h) = p_0(h), \qquad z(0)=z_0.\end{equation*}

In the following, we will consider the model \ref{sistema1}) and prove its well-posedness. To this end, tools from functional analysis are employed.

In this paper, we want to prove existence and uniqueness of a nonnegative solution for Problem (\ref{sistema1}) for both cases of operator $L_h$ given in (\ref{Lh1})-(\ref{Lh2}), in a $L^2$ framework, and where $g$ is a  positive functional satisfying some Lipschitz property and is positive up to a translation (\textit{i.e.} $(g+\lambda I)$ is positive for some $\lambda$). To achieve that goal, we follow a standard line of proof, sketched next with an outline of the changes and difficulties
encountered.
We rewrite the model as a Cauchy problem, we prove that the linear part generates a positive $C_0$-semigroup, we check that the nonlinear part verifies a Lipschitz property and is positive up to a translation.

These latter points then allow us to use a fixed point theorem to get the desired result. Such kind of mathematical developments have already been published for PDE structured models (\cite{Magal09}, \cite{PerassoRazafison2013}, \cite{PerassoRichard2017}), reaction diffusion systems (\cite{Amann05}, \cite{Apreutesei2010}, \cite{Apreutesei2014}, \cite{DuprezPerasso2017}) and a case mixing diffusion and age-structure \cite{Walker08}. 

In the present model, some new technical difficulties appear, due to the shape of the system. First, there is a nonhomogeneous Dirichlet boundary condition for $n$, so we need to make the change of variables:
\begin{equation}
\tilde{n}=n-n_H \label{Eq:Change_Var}
\end{equation}
in order to get a Cauchy problem. Consequently, in addition to the proof that the linear part generates a positive $C_0$-semigroup, we also need to prove a lower bound property, implying that $\{f\in L^2(0,H): f(x)\geq -n_H \quad a.e. \quad x\in[0,H]\}$ is invariant under the semigroup. Moreover, another critical point in the mathematical analysis stands in a singularity of the nonlinear part at:
$$n=-1/\chi$$
so we need to restrict the space to a subset where the denominator is nonzero. A final difficulty is that the nonlinear part does not satisfy the required Lipschitz property in $L^2$, but does in $L^\infty$. Consequently, we need some $L^\infty$ estimates, that are proved by using the truncation method of Stampacchia (see \textit{e.g.} \cite{Brezis99}).
\smallskip

The paper is structured as follows: in the next section, we make explicit the framework used in the sequel, taking into account the model specificities as previously described. Section \ref{sec:wellpos}, dealing with well-posedness, is dedicated to the main results of the article; we first prove that the linear part generates a $C_0$-semigroup that satisfies some lower and upper bounds; we then handle the nonlinear part showing it satisfies a Lipschitz property and checking that it is positive up to a translation, implying the existence and uniqueness of a nonnegative solution; we then show that the solution is global since it cannot explode in finite time, prove that $n$ is bounded and give a sufficient condition to get extinction for $p$ and $z$. All these results are obtained for the two cases of operator $L_h$ as defined in (\ref{Lh1})-(\ref{Lh2}).

\section{Framework}

For sake of simplicity and without loss of generality, we assume in all that follows that the diffusion coefficient is $D=1$. Recalling (\ref{Eq:Change_Var}), it follows that the model (\ref{sistema1}) becomes
\begin{equation}\label{sistema3}
\left\{
\begin{array}{rcl}
\frac{\partial \tilde{n}}{\partial t}&=&\frac{\partial^2 \tilde{n}}{\partial h^2} - L_h(p) \left(\frac{\tilde{n}+n_H}{1+\chi (\tilde{n}+n_H)}\right), \\
\frac{\partial p}{\partial t}&=&\frac{\partial^2 p}{\partial h^2}+ L_h(p) \left(\frac{\tilde{n}+n_H}{1+\chi (\tilde{n}+n_H)}\right) -z g(p)-m_pp, \\
\frac{dz}{dt}&=&\frac{kz(t)}{H} \displaystyle \int_0^H g(p)(t,h) \,dh -mz(t),
\end{array}
\right.
\end{equation}
for every $t\geq 0$, $h\in [0,H]$, with the boundary conditions:
\begin{eqnarray*}
&\frac{\partial p}{\partial h} (t,0) =0, \qquad  \frac{\partial p}{\partial h} (t,H) = 0, \\
&\frac{\partial\tilde{n}}{\partial h} (t,0) = 0 , \qquad \tilde{n}(t,H)=0.	
\end{eqnarray*}
Since $n=\tilde{n}+n_H$, it suffices to prove that the problem (\ref{sistema3}) is well-posed in a suitable Banach space, in the semigroups setting. We will then drop the tilde in the following and write $n$ instead of $\tilde{n}$, for a better reading. We work in the Hilbert space
$$\X=(L^2(0,H)\times L^2(0,H)\times \R,\|\cdot \|_{\X}),$$
endowed with the norm
$$\|(n,p,z)\|_{\X}=\|n\|_{L^2(0,H)}+\|p\|_{L^2(0,H)}+|z|$$
and the scalar product 
$$\left\langle (n_1,p_1,z_1), (n_2,p_2,z_2)\right \rangle_{\X}=\left\langle
n_1, n_2\right\rangle_{L^2(0,H)}+\left\langle
p_1, p_2\right\rangle_{L^2(0,H)}+z_1z_2.$$
We define the linear operator $\A:D(\A)\subset \X\to \X$ by:
$$\A\left(\begin{array}{cc}
n \\
p \\
z
\end{array}\right)=\left(\begin{array}{cc}
n'' \\
p''-m_pp \\
-m z
\end{array}\right),$$
with domain $D(\A)$ given by
$$\{(n,p,z)\in H^2(0,H)\times H^2(0,H)\times \R: n'(0)=n(H)=p'(0)=p'(H)=0\}.$$
Note here that $(n,p,z)$ belong to $D(\A)\subset \X$ and are time-independent, while it was a function of time (and space) in the model (\ref{sistema1}). The derivatives are consequently taken with respect to $h\in [0,H]$, \textit{e.g.}
$$
n'=\frac {dn}{dh}, \qquad p'=\frac {dp}{dh}.
$$
For sake of simplicity we keep the same notations, though the space will always be specified to avoid some possible confusion.

Since we are interested in the positivity of the solutions, we denote by $\X_+$ the positive cone of $\X$. Actually, because of the change of variable (\ref{Eq:Change_Var}), we have
$$n\geq 0 \Longleftrightarrow \tilde{n}\geq -n_H,$$
where $n$ and $\tilde{n}$ are respectively the solutions of (\ref{sistema1}) and (\ref{sistema3}). To this end we define, for every $\ep\geq 0$, the space
$$\X_{\ep}:=\{(n,p,z)\in \X: (n+\ep \mathbf{1}_{[0,H]},p,z)\in \X_+\}.$$
We see that $\X_0=\X_+$ and the sequence of spaces $\{\X_{\ep}\}_{\ep\geq 0}$ is increasing in the sense that
$$\X_+\subset \X_{\ep_1}\subset \X_{\ep_2}, \quad \forall \ep_2\geq \ep_1\geq 0.$$
We will then obtain the positivity when considering $\ep=n_H$. We now suppose, and in all that follows, that the nonlinear functional $g$ satisfies the assumption below.
\begin{assumption}\label{Assump}
We suppose that
$g:L^\infty_+(0,H)\to L^\infty_+(0,H)$ and there exists $\lambda>0$ such that $\lambda p-g(p)\in L^\infty_+(0,H)$ for every $p\in L^\infty_+(0,H)$, and for every $m>0$ there exists some constant $l_m\geq 0$ such that for every $(p_1,p_2)\in \{p\in L^\infty_+(0,H): \|p\|_{L^\infty(0,H)}\leq m\}^2$, we have
$$\|g(p_1)-g(p_2)\|_{L^\infty(0,H)}\leq l_m\|p_1-p_2\|_{L^\infty(0,H)}$$
\end{assumption}

\begin{remark}
We can note that the classical functional responses Holling types I,II II, Ivlev and the one given in \cite{AMV}:
$$g(p)=p, \quad g(p)=\frac{p}{1+p}, \quad g(p)=\frac{p^2}{1+p^2}, \quad g(p)=(1-e^{-p}), \quad g(p)=\frac{p^2}{(1+p)}$$
satisfy the Assumption \ref{Assump}.
\end{remark}
Since the functional $g$ is defined in $L^\infty(0,H)$, we need to define the Banach space 
$$\X^\infty=(L^\infty(0,H)\times L^\infty(0,H)\times \R,\|\cdot\|_{\X^\infty})\subset \X$$
endowed with the norm
$$\|(n,p,z)\|_{\X^\infty}=\|n\|_{L^\infty(0,H)}+\|p\|_{L^\infty(0,H)}+|z|$$
and we also define $\X^\infty_+$ the positive cone of $\X^\infty$, as well as the spaces
$$\X^\infty_{\ep}:=\{(n,p,z)\in \X^\infty: (n+\ep \mathbf{1}_{[0,H]},p,z)\in \X^\infty_+\}\subset \X^\infty,$$
for every $\ep\geq 0$. Because of the singularity of the nonlinear part in (\ref{sistema3}) at
$$-n_H-\frac{1}{\chi}$$
we define, according to the two cases of operator $L_h$ given in (\ref{Lh1})-(\ref{Lh2}), the functions $f_i:\X^\infty_{n_H+(2\chi)^{-1}}\to \X^\infty$ by:
$$f_{1}(n,p,z)=\left(\begin{array}{cc}
-r \exp(-\gamma \cdot)p \left( \frac{n+n_H}{1+\chi (n+n_H)}\right) \\
r \exp(-\gamma \cdot) p \left(\frac{n+n_H}{1+\chi (n+n_H)}\right) -g(p) \vspace{0.1cm} \\
\frac{kz}{H} \displaystyle \int_0^H g(p)(t,h) \,dh
\end{array}\right),$$
$$f_{2}(n,p,z)=\left(\begin{array}{cc}
-r\exp(-\nu \int_0^{h}p(x)\,dx) p\left(\frac{n+n_H}{1+\chi (n+n_H)}\right) \\
r\exp(-\nu \int_0^{h}p(x)\,dx) p\left(\frac{n+n_H}{1+\chi (n+n_H)}\right) -g(p) \vspace{0.1cm} \\
\frac{kz}{H} \displaystyle \int_0^H g(p)(t,h) \,dh
\end{array}\right)$$
for each $i\in\{1,2\}$.

\begin{lemma}
The range of $f_i$ is included in $\X^\infty$ for each $i\in\{1,2\}$.
\end{lemma}

\begin{proof}
Let $(n,p,z)\in \X^\infty_{n_H+(2\chi)^{-1}}$, then
$$\begin{array}{rcl}
\|f_{1}(n,p,z)\|_{\X^\infty} \leq \frac{2r}{\chi}\|p\|_{L^\infty}+\|g_2(p)\|_{L^\infty}+k|z|\|g_2(p)\|_{L^\infty}<\infty.
\end{array}$$
and the same inequality holds for $f_{2}$.
\end{proof}
When focusing on (\ref{sistema3}), we will consequently study thereafter the following abstract Cauchy problems:
\begin{equation}\label{Eq:CauchyPb}
\begin{array}{rcl}
	U'(t)&=&\A U(t) +f_{i}(U(t)), \quad \forall t>0, \quad \textnormal{in }  \X^\infty_{n_H}, \\
	U(0)&=&U_0 \in \X^\infty_{n_H}\subset X^\infty_{n_H+(2\chi)^{-1}},
\end{array}
\end{equation}
for every $i\in\{1,2\}$, where $U(t)=(n(t),p(t),z(t))^T$. The approach used to prove existence and uniqueness of a solution of (\ref{Eq:CauchyPb}) is classical (see \textit{e.g.} \cite{Pazy83}). The techniques used for both models are the same: we first show that $\A$ generates a $C_0$-semigroup in $\X$, then we prove some Lipschitz property for $f_{i}$.
Now that the framework is clear, we can deal with the well-posedness of the Cauchy problem (\ref{Eq:CauchyPb}).

\section{Well-posedness}\label{sec:wellpos}

\subsection{Linear part}

We start this section by handling the linear part.

\begin{theorem}\label{Thm:Generation}
For every $\nu\geq 0$, the operator $\A-\nu I$ generates a $C_0$-semigroup $\{T_{\A-\nu I}(t)\}_{t\geq 0}$ on $\X$. Moreover it satisfies
\begin{equation}\label{Eq:Lin_C1}
\forall u_0 \in \X, \quad t\longmapsto T_{\A-\nu I}(t)u_0 \in \Co([0,\infty),\X)\cap \Co^1((0,\infty),\X),
\end{equation}
\begin{equation}
\label{Eq:Sg_Bound}
\|T_{\A-\nu I}(t)u_0\|_{\X^\infty}\leq \|u_0\|_{\X^\infty}, \quad \forall t\geq 0, \quad \forall u_0\in \X^\infty,
\end{equation}
whence $\{T_{\A-\nu I}(t)\}_{t\geq 0}\subset \mathcal{L}(\X^\infty)$ and
\begin{equation}
\label{Eq:Sg_Pos}T_{\A-\nu I}(t)u_0 \in \X_{\ep}, \quad \forall t\geq 0, \quad \forall \ep\geq 0, \quad \forall u_0 \in \X_\ep.
\end{equation}
Note that (\ref{Eq:Sg_Pos}) implies the positivity of $\{T_{\A-\nu I}(t)\}_{t\geq 0}$.
\end{theorem}

\begin{proof}
The sketch of the proof is the following: we first prove that $\A-\nu I$ generates a $C_0$-semigroup by verifying the surjectivity and the dissipativity properties. We deduce that for every initial condition $(n_0,p_0,z_0)\in \X$, the solution of the linear problem verifies (\ref{Eq:Lin_C1}). We then show that this solution (denoted by $(n,p,z)$) verifies the following inequalities:
\begin{equation} \label{Eq:Maj_n}
\min\{0,\inf_{h\in[0,H]}n_0(h)\} \leq n(t,h)\leq \max\{0,\sup_{h\in[0,H]}n_0(h)\},
\end{equation}
\begin{equation}\label{Eq:Maj_p}
\min\{0,\inf_{h\in[0,H]}p_0(h)\}\leq p(t,h) \leq \max\{0,\sup_{h\in[0,H]}p_0(h)\},
\end{equation}
\begin{equation}\label{Eq:Maj_z}
-|z_0|\leq z(t)\leq |z_0|,
\end{equation}
for every $t\geq 0$, a.e. $h\in [0,H]$ and (\ref{Eq:Sg_Bound}) follows. We then check that $\{T_{\A-\nu I}(t)\}_{t\geq 0}$ is positive. Finally we prove (\ref{Eq:Sg_Pos}).
\begin{enumerate}
\item Clearly, $D(\A)$ is dense into $\X$. Moreover, for every $(n,p,z)\in D(\A)$, we have
\begin{eqnarray*}
\hspace{-1.5cm}
\begin{array}{ll}
&\left\langle \A(n,p,z), (n,p,z)\right\rangle_{\X}\\
&=\left\langle Dn'',n\right\rangle_{L^2}+\left\langle Dp''-m_p p,p\right\rangle_{L^2}-mz^2 \\
&=D\displaystyle \int_0^H n(h)\frac{\partial^2 n}{\partial h^2}dh+D\int_0^H p(h)\frac{\partial^2 p}{\partial h^2}dh-m_p \int_0^H p(h)^2 dh-mz^2 \\
&=-D\displaystyle \int_0^H \left(\frac{\partial n}{\partial h}\right)^2dh-D\int_0^H \left(\frac{\partial p}{\partial h}\right)^2dh-m_p \int_0^H p(h)^2 dh-mz^2 \\
&\leq 0.
\end{array}
\end{eqnarray*}
Consequently, $\A$ is dissipative in $\X$. Let us show now that $\lambda I-\A:D(\A)\to \X$ is surjective for any $\lambda>0$. Let $\overline{H}=(h_n,h_p,h_z)\in \X$ and $\lambda>0$. We look for $U:=(n,p,z)^T \in D(\A)$ such that $(\lambda I-\A)U=\overline{H}$, \textit{i.e.} 
\begin{eqnarray}
 \lambda n-n''=h_n, \label{Eq:Surj1} \\
 \lambda p -p''+m_p p=h_p, \label{Eq:Surj2} \\
 \lambda z+m z=h_z, \nonumber
\end{eqnarray}
so
$$z=\frac{h_z}{\lambda+m}.$$
We multiply (\ref{Eq:Surj1}) and (\ref{Eq:Surj2}) respectively by $u\in H^1(0,H)$ and $v\in H^1(0,H)$, then integrate between $0$ and $H$ to get
\begin{eqnarray*}
\left\{
\begin{array}{rcl}
\displaystyle \lambda \int_0^H nu-\int_0^H n''u&=&\displaystyle \int_0^H h_n u,\\
\displaystyle \lambda \int_0^H pv-\int_0^H p''v+m_p \int_0^H pv&=&\displaystyle \int_0^H h_p v.
\end{array}
\right.
\end{eqnarray*}
An integration by parts gives
\begin{eqnarray}
 \lambda \int_0^H nu+\int_0^H n'u'=\int_0^H h_n u, \label{Eq:Lax1} \\
 \lambda \int_0^H pv+\int_0^H p'v'+m_p \int_0^H pv=\int_0^H h_p v, \label{Eq:Lax2}
\end{eqnarray}
whence
$$a_1(n,u)=L_1(u), \quad a_2(p,v)=L_2(v),$$
where the bilinear forms $a_1:V\times V\to \R$, $a_2:H^1(0,H)\times H^1(0,H)\to \R$ and the linear forms $L_1:V\to \R$, $L_2:H^1(0,H)\to \R$ are defined by: 
$$a_1(n,u)=\lambda \int_0^H nu+\int_0^H n'u',$$
$$a_2(p,v)=\lambda \int_0^H pv+\int_0^H p'v'+m_p \int_0^H pv,$$
$$L_1(u)=\int_0^H h_n u, \qquad L_2(v)=\int_0^H h_p v,$$
where 
$$V:=\{u\in H^1(0,H): u(H)=0\}.$$
A simple application of Lax-Milgram theorem implies that for every $(h_n,h_p) \in  (L^2(0,H))^2$, there exists a unique $(n,p)\in V\times H^1(0,H)$ such that:
\begin{eqnarray*}
\left\{
\begin{array}{rcl}
a_1(n,u)&=&L_1(u), \\
a_2(p,v)&=&L_2(v),
\end{array}
\right.
\end{eqnarray*}
for every $(u,v)\in V\times H^1(0,H)$.

Now, we verify that $U$ belongs to $D(\A)$. For this, we use (\ref{Eq:Lax1}) and (\ref{Eq:Lax2}) with $u\in \Co ^\infty_c([0,H])$ and $v\in \Co ^\infty_c([0,H])$ respectively, where $\Co _c^\infty(0,H)$ refers to $\Co^\infty$ functions with compact support. Then, we get
$$\left| \int_0^H n'u' \right|\leq [|\lambda|\|n\|_{L^2(0,H)}+\|h_n\|_{L^2(0,H)}]\|u\|_{L^2(0,H)}\leq c_1 \|u\|_{L^2},$$
$$\left| \int_0^H p'v' \right|\leq [(|\lambda|+|m_p|)\|p\|_{L^2(0,H)}+\|h_p\|_{L^2(0,H)}]\|v\|_{L^2(0,H)}\leq c_2 \|v\|_{L^2},$$
for some constant $c_1$ and $c_2$. Consequently $n'\in H^1(0,H)$ and $p'\in H^1(0,H)$, so $n\in H^2(0,H)$ and $p\in H^2(0,H)$. Finally, to prove the surjectivity, an integration by parts of (\ref{Eq:Lax1})-(\ref{Eq:Lax2}) with $u\in \Co _c(0,H)$ and $v\in \Co _c(0,H)$ implies (\ref{Eq:Surj1}) and (\ref{Eq:Surj2}). Moreover, an integration by parts of (\ref{Eq:Lax1}) with $u\in \Co (0,H)$, $u(0)=1$, $u(H)=1$ implies that $n'(0)=0$. Similarly, we get $p'(0)=0$ and $p'(H)=0$ after an integration by parts of (\ref{Eq:Lax2}) with $v\in \Co (0,H)$ and respectively $v(0)=1, v(H)=0$ and $v(0)=0, v(H)=1$. Thus $\A$ generates a $C_0$-semigroup $\{T_{\A}(t)\}_{t\geq 0}$ by Lumer-Phillips theorem, and $\A-\nu I$ also generates a $C_0$-semigroup $\{T_{\A-\nu I}(t)\}_{t\geq 0}$ for every $\nu\geq 0$ by bounded perturbation arguments.

\item Let $\nu\geq 0$. We readily see that $\A-\nu I$ is a symmetric operator. It is actually a self-adjoint operator since it is $m$-dissipative (with \cite{Brezis99}, Proposition VII.6, p. 113). Using \cite{Brezis99}, Theorem VII.7, p. 113, we obtain that the solution of
\begin{equation}
\left\{
\begin{array}{rcl}
U'(t)&=&(\A-\nu I)U(t) \\
U(0)&=&u_0 \in \X
\end{array}
\right.
\end{equation}
verifies (\ref{Eq:Lin_C1}).

\item Let $\nu\geq 0$. We want to prove that the solution $U(t):=(n(t,\cdot), p(t,\cdot), z(t))$ of
\begin{equation}\label{Eq:Lin_Sys}
\left\{
\begin{array}{rcl}
U'(t)&=&(\A-\nu I)U(t) \\
U(0)&=&(n_0,p_0,z_0)\in \X
\end{array}
\right.
\end{equation}
verifies (\ref{Eq:Maj_n})-(\ref{Eq:Maj_p})-(\ref{Eq:Maj_z}), for every $t\geq 0$. It is clear that
$$z(t)=z_0 e^{-(\nu+m)t}$$
so that (\ref{Eq:Maj_z}) is satisfied for every $t\geq 0$. To get the result on $n$ and $p$, we use the truncation method of Stampacchia (see \textit{e.g.} \cite{Brezis99}, Theorem X.3, p. 211). In all the following, we will use the notation
$$ \KK^\sigma:=\max\{0,\sup_{h\in[0,H]}\sigma(h)\}\geq 0, \qquad \KK_\sigma:=-\min\{0,\inf_{h\in[0,H]}\sigma(h)\}\geq 0$$
for every function $\sigma \in L^\infty(0,H)$. Define the function $G\in \Co ^1(\R)$ such that
\begin{enumerate}
\item $|G'(x)|\leq M, \quad \forall x \in \R$,
\item $G$ is strictly increasing on $(0,\infty)$, 
\item $G(x)=0, \quad \forall x\leq 0$.
\end{enumerate}
We introduce the functions
\begin{equation} \label{Eq:Kappa}
\kappa:x\mapsto\int_0^x G(\sigma)d\sigma, \quad \forall x\in \R,
\end{equation}
$$\varphi_1:t\mapsto\int_0^H\kappa(p(t,h)-\KK^{p_0})dh, \ \ \varphi_2:t\mapsto\int_0^H\kappa(\overline{p}(t,h)-\KK_{p_0})dh, \ \ \forall t\geq 0,$$
$$\varphi_3:t\mapsto\int_0^H\kappa(n(t,h)-\KK^{n_0})dh, \ \ \varphi_4:t\mapsto\int_0^H\kappa(\overline{n}(t,h)-\KK_{n_0})dh, \ \ \forall t\geq 0,$$
where 
$$\overline{p}:=-p, \qquad \overline{n}:=-n.$$
Define the set
$$
\Y:=\{\varphi\in\Co ([0,\infty),\R), \ \ \varphi(0)=0, \ \ \varphi \geq 0 \ \textnormal{ on } \ [0,\infty),\ \ \varphi \in \Co ^1((0,\infty),\R)\}.
$$
We can show that $\varphi_i \in \Y$ for every $i\in \llbracket 1,4\rrbracket$, using (\ref{Eq:Lin_C1}). Moreover, we have
\begin{eqnarray*}
\varphi_1'(t)&=&\displaystyle \int_0^H G(p(t,h)-\KK^{p_0})\frac{\partial p}{\partial t}(t,h)dh \\
&=&\displaystyle \int_0^H G(p(t,h)-\KK^{p_0})\left(\frac{\partial^2 p}{\partial h^2}(t,h)-(\nu+m_p) p(t,h)\right)dh \\
&=&-\displaystyle \int_0^H G'(p(t,h)-\KK^{p_0})\left|\frac{\partial p}{\partial h}(t,h)\right|^2 dh \vspace{0.1cm} \\
&&-\displaystyle \int_0^H G (p(t,h)-\KK^{p_0})(\nu+m_p) p(t,h)dh \leq 0, \quad \forall t>0,
\end{eqnarray*}
since $G'\geq 0$. Finally $\varphi_1'\leq 0$ on $(0,\infty)$ and consequently $\varphi_1\equiv 0$, so
$$p(t,h)\leq \KK^{p_0}\leq \max\{0,\sup_{h\in[0,H]}p_0(h)\}, \quad \forall t\geq 0, \quad \textnormal{a.e. } h\in [0,H].$$
The same computations lead to
\begin{eqnarray*}
\begin{array}{rcl}
\varphi_2'(t)&=&-\displaystyle \int_0^H G'(\overline{p}(t,h)-\KK_{p_0})\left|\frac{\partial \overline{p}}{\partial h}(t,h)\right|^2 dh \vspace{0.1cm} \\
&&-\displaystyle \int_0^H G (\overline{p}(t,h)-\KK_{p_0})(\nu+m_p) \overline{p}(t,h)dh \leq 0
\end{array}
\end{eqnarray*}
for every $t>0$ and $\varphi_2 \equiv 0$ on $(0,\infty)$, so
$$p(t,h)\geq -\KK_{p_0}\geq \min\{0,\inf_{h\in[0,H]}p_0(h)\}, \quad \forall t\geq 0, \quad \textnormal{a.e. } h\in[0,H]$$
and (\ref{Eq:Maj_p}) is satisfied. Similarly, we have
\begin{eqnarray*}
\begin{array}{rcl}
\varphi_3'(t)&=&\displaystyle \int_0^H G (n(t,h)-\KK^{n_0})\frac{\partial n}{\partial t}(t,h)dh \\
&=&\displaystyle \int_0^H G (n(t,h)-\KK^{n_0})\left(\frac{\partial^2 n}{\partial h^2}(t,h)-\nu n(t,h)\right)dh \\
&=&-\displaystyle \int_0^H G'(n(t,h)-\KK^{n_0})\left|\frac{\partial n}{\partial h}(t,h)\right|^2 dh \\
&&-\displaystyle \int_0^H G(n(t,h)-\KK^{n_0})\nu n(t,h)dh \leq 0, \quad \forall t>0,
\end{array}
\end{eqnarray*}
since $G(n(t,H)-\KK_{n_0})=G(-\KK_{n_0})=0$. We can also show that
$$\varphi_4'(t)\leq 0, \quad \forall t>0$$
whence (\ref{Eq:Maj_n}) holds. Considering an initial condition $(n_0,p_0,z_0)\in \X^\infty$ leads easily to (\ref{Eq:Sg_Bound}).

\item Let us prove now that $\{T_{\A-\nu}(t)\}_{t\geq 0}$ is positive for every $\nu \geq0$, that is, the resolvent
$$R_{\lambda}(A-\nu I):=((\lambda+\nu)I-\A)^{-1}$$
is positive for $\lambda$ large enough (see \textit{e.g.} \cite{Clement87}, p. 165). Let $\nu \geq 0$, $\lambda\geq 0$, $\overline{H}:=(h_n,h_p,h_z)\in \X_+$. As point 1. above, one can consider
$$U:=(n,p,z)=(R_\lambda(\A-\nu I))\overline{H}\in D(\A).$$
We have to prove that $U\in \X_+$. Since $\Co ([0,H])$ is dense in $L^2(0,H)$, we may assume without loss of generality (using the dissipativity and the closedness of $\A$) that
$$h_n\in \Co ([0,H]), \qquad h_p\in \Co ([0,H]).$$
Thus, we have
$$-p''+(\lambda+\nu+m_p)p=h_p,$$
with $p\in H^2(0,H)\subset \Co ([0,H])$. Since $h_p$ is continuous, then the latter equation implies that $p''$ is also continuous and then $p\in \Co ^2([0,H])$. The absolute minimum of $p$ is achieved at some $\overline{h}\in[0,H]$. Suppose that $p(\overline{h})<0$. The function 
$$q:=-p$$
verifies the equation
$$q''-(\lambda+\nu+m_p)q=h_p\geq 0,$$
and its absolute maximum is reached at $\overline{h}$. 
If $\overline{h}=0$, then by Hopf's maximum principle (see \cite{ProtterWeinberger84}, Theorem 4, p. 7), we would have
$$-p'(0)=q'(0)>0,$$
which contradicts the Neumann boundary condition. If $\overline{h}=H$ then by Hopf's maximum principle, we would have
$$-p'(H)=q'(H)<0,$$
which is absurd. Finally, if $\overline{h}\in(0,H)$ then 
$$0\geq -p''(\overline{h})=h_p(\overline{h})-(\lambda+\nu+m_p)p(\overline{h})>0$$
which is not possible. Consequently
$$p(h)\geq p(\overline{h})\geq 0, \ \forall h\in[0,H].$$
Similarly, $n\in \Co ^2([0,H])$ verifies the equation
$$-n''+(\lambda+\nu)n=h_n\geq 0.$$
Moreover, $n$ reaches its absolute minimum at $\overline{h}\in[0,H]$. If $n(\overline{h})<0$, then the same arguments than before lead to
$$\overline{h}=H,$$
which contradicts the fact that $n(H)=0$.
Consequently 
$$n(h)\geq n(\overline{h})\geq 0, \quad \forall h\in[0,H].$$
Finally, it is clear that
$$z=\frac{h_z}{\lambda+\nu+m_p}\geq 0,$$
which proves that $R_\lambda(\A+\nu I)$ is positive and consequently that the $C_0$-semigroup $\{T_{\A-\nu I}(t)\}_{t\geq 0}$ is positive for every $\nu\geq 0$.

\item Now we want to prove (\ref{Eq:Sg_Pos}). Let $\ep \geq 0$, $\nu \geq 0$, $(n_0,p_0,z_0)\in \X_{\ep}$ and $(n,p,z)$ the solution of (\ref{Eq:Lin_Sys}). Because of the positivity of $\{T_{\A-\nu I}(t)\}_{t\geq 0}$, it only remains to prove that
\begin{eqnarray*}
n(t,h)\geq -\ep, \quad \forall t\geq 0, \quad \textnormal{a.e. } h \in[0,H]
\end{eqnarray*}
which arises from (\ref{Eq:Maj_n}).
\end{enumerate}
\end{proof}

\subsection{Nonlinear part}

In this section we handle the nonlinear part by showing a Lipschitz and a positivity properties of $f_{i}$ for each $i\in\{1,2\}$. Let $m>0$, then define the set
$$B_m:=\{(n,p,z)\in \X^\infty: \|(n,p,z)\|_{\X^\infty}\leq m\}.$$

\begin{proposition}\label{Prop:Lips}
For every $m>0$, there exists some constant $k_m\geq 0$ such that for every $((n_1,p_1,z_1),(n_2,p_2,z_2))\in \left(\X^\infty_{n_H+(2\chi)^{-1}} \cap B^\infty_m\right)^2$, we have
$$\left\|f_{i}\left(\begin{array}{cc}
n_2 \\
p_2 \\
z_2
\end{array}\right)-f_{i}\left(\begin{array}{cc}
n_1 \\
p_1 \\
z_1
\end{array}\right)\right\|_{\X^\infty}\leq k_m \left\|\left(\begin{array}{cc}
n_2 \\
p_2 \\
z_2
\end{array}\right)-\left(\begin{array}{cc}
n_1 \\
p_1 \\
z_1
\end{array}\right)\right\|_{\X^\infty}.$$
\end{proposition}

\begin{proof}
We first prove the result for $f_{1}$, the case $f_{2}$ being similar.\\
Let $((n_1,p_1,z_1),(n_2,p_2,z_2))\in \left(\X_{n_H+(2\chi)^{-1}} \cap B_m\right)^2$. Some computations give
\begin{eqnarray*}
\begin{array}{rcl}
&&\left\|f_{1}\left(\begin{array}{cc}
n_2, p_2, z_2
\end{array}\right)^T-f_{1}\left(\begin{array}{cc}
n_1, p_1, z_1
\end{array}\right)^T \right\|_{\X^\infty} \\
&\leq & 2r\left\|\frac{p_2(n_2+n_H)}{1+\chi(n_2+n_H)}-\frac{p_1(n_1+n_H)}{1+\chi(n_1+n_H)}\right\|_{L^{\infty}}+\left\|z_1 g(p_1)-z_2 g(p_2)\right\|_{L^\infty}\\
&&+\frac{k}{H}\displaystyle \int_0^H \left\| z_1 g(p_1)-z_2 g(p_2)\right\|_{L^\infty}dh \\
&\leq& 2r\left(m\|n_2-n_1\|_{L^\infty}(1+\chi (m+n_H))+(m+n_H)\|p_2-p_1\|_{L^\infty}\right) \\
&&+\left(m l_m\|p_1-p_2\|_{L^\infty}+m l_m |z_1-z_2|\right)\left(1+k\right)
\end{array}
\end{eqnarray*}
by Assumption \ref{Assump}, which proves the result.
\end{proof}

\begin{proposition}\label{Prop:Pos}
For every $m>0$, there exists $\lambda_m\geq 0$ and $\eta_m\geq 0$ such that for every $(n,p,z)\in \X^\infty_{n_H+(2\chi)^{-1}}\cap B_m$, we have
$$f_i(n,p,z)+\lambda_m(n,p,z)\in \X^\infty_{\eta_m}.$$
\end{proposition}

\begin{proof}
Let $(n,p,z)\in \X^\infty_{n_H+(2\chi)^{-1}}\cap B_m$, then
\begin{eqnarray*}
&f_1(n,p,z)+\lambda_m(n,p,z) \\
&=\left(\begin{array}{cc}
n\left(\lambda_m-r\exp(-\gamma \cdot)\frac{p}{1+\chi (n+n_H)}\right)-r\exp(-\gamma \cdot)\frac{p n_H}{1+\chi(n+n_H)} \\
p\left(\lambda_m+ r \exp(-\gamma \cdot)  \frac{n+n_H}{1+\chi (n+n_H)}\right)-zg(p) \\
z\left(\lambda_m+  \frac{k}{H} \displaystyle \int_0^H g(p)(t,h)dh\right)
\end{array}\right).
\end{eqnarray*}
Note that by Assumption \ref{Assump}, there exists $\lambda>0$ such that $\lambda p-g(p)\geq 0$, so choosing $\lambda_m\geq m \lambda$ induces that $p \lambda_m-zg(p)\geq m(\lambda p-g(p))\geq 0$. Consequently, it suffices to consider
\begin{equation}
\label{Eq:Lambda_m}
\lambda_m=m\lambda
\end{equation}
and
\begin{equation}
\label{Eq:Eta_m}
\eta_m=m\lambda_m+rm^2+r m n_H
\end{equation}
which ends the proof.
\end{proof}

\subsection{Local existence and positivity}

We are now able to show existence and uniqueness of a solution.

\begin{theorem}\label{Thm:Exist_Loc}
Suppose that operator $L_h$ has one of the shapes given in (\ref{Lh1}) or in (\ref{Lh2}). Then for every initial condition $(n_0,p_0,z_0)\in \X^\infty_{n_H}$, there exists a unique solution $(n,p,z)\in \Co \left([0,t_{\max}),\X^\infty_{n_H}\right)$ for the system (\ref{sistema3}), where $t_{\max}\leq \infty$.
\end{theorem}

\begin{proof}
Let $(n_0,p_0,z_0)\in \X^\infty_{n_H}$ and
$$m=2\|(n_0,p_0,z_0)\|_{\X^\infty}.$$
Define the constants $\lambda_m\geq 0$, $\eta_m\geq 0$ respectively by (\ref{Eq:Lambda_m}) and (\ref{Eq:Eta_m}), the linear operator
$$\A_m=\A-\lambda_m I:D(\A)\subset \X\to \X,$$
and for $i=1,2$ the nonlinear function 
$$f_m=f_i+\lambda_m I:\X^\infty_{n_H+(2\chi)^{-1}}\to \X.$$
We readily see that $\A_m$ is the infinitesimal generator of a $C_0$-semigroup $\{T_{\A_m}(t)\}_{t\geq 0}$ on $\X$. Let 
$$
\tau=\min\left\{
\frac{1}{2(k_m+\lambda_m)},\frac{1}{2\chi \eta_m}\right\}
>0.
$$
A consequence of Theorem \ref{Thm:Generation} and Proposition \ref{Prop:Lips} is that the nonlinear operator
$$G:\Co \left([0,\tau],\X^\infty_{n_H+(2\chi)^{-1}}\right)\to \Co ([0,\tau],\X)$$
defined by
\begin{equation}\label{Eq:G}
G\left(\begin{array}{cc}
n(t,\cdot) \\
p(t,\cdot) \\
z(t)
\end{array}\right)
=T_{\A_m}(t)\left(\begin{array}{cc}
n_0 \\
p_0 \\
z_0
\end{array}\right)+\displaystyle \int_0^t T_{\A_m}(t-s)f_m\left(\begin{array}{cc}
n(s,\cdot) \\
p(s,\cdot) \\
z(s)
\end{array}\right)ds
\end{equation}
is a $1/2$-shrinking operator on
$$\ZZ:=\Co \left([0,\tau],\X^\infty_{n_H+(2\chi)^{-1}}\cap B_m\right)$$
with $G(\ZZ)\subset B_m$, since 
$$t\leq \tau\leq \frac{1}{2(k_m+\lambda_m)}.$$
Moreover, using Theorem \ref{Thm:Generation}, the fact that
$$\tau\leq \frac{1}{2\chi \eta_m},$$
and Proposition \ref{Prop:Pos}, then
$$G\left(\begin{array}{cc}
n(t,\cdot) \\
p(t,\cdot) \\
z(t)
\end{array}\right)\in \X_{n_H+(2\chi)^{-1}} \quad \forall t\in [0,\tau].$$
Consequently $G$ preserves the space $\ZZ$. The Banach-Picard theorem then implies the existence and uniqueness of a local solution 
$$(n,p,z)\in \Co \left([0,\tau],\X^\infty_{n_H+(2\chi)^{-1}}\cap B_m\right).$$
It remains to prove that
\begin{equation}\label{Eq:Bound_n}
n(t,h) \geq -n_H, \quad \forall t\in[0,\tau], \quad \forall h\in[0,H].
\end{equation}
First, suppose that
\begin{equation}\label{Hyp:CI}
(n_0,p_0,z_0)\in D(\A)\cap \X^\infty_{n_H}.
\end{equation}
Using \cite[Theorem 6.1.7, p. 190]{Pazy83}, the solution $(n,p,z)$ of (\ref{sistema3}) is classical. Consequently, the function
$$\overline{n}:=-n$$
satisfies the equation
$$\frac{\partial \overline{n}}{\partial t}(t,h)=\frac{\partial^2 \overline{n}}{\partial h^2}(t,h)+L_h(p)(t,h)\left(\frac{n_H-\overline{n}(t,h)}{1+\chi(n+n_H)}\right),$$
for every $t\in(0,\tau]$ and a.e. $h\in[0,H]$.
Define the function
$$\varphi_{\overline{n}}(t)=\int_0^H \kappa(\overline{n}(t,h)-n_H)dh,$$
where $\kappa$ is given by (\ref{Eq:Kappa}), for every $t\in (0,\tau]$. We can check that
$$\varphi_{\overline{n}}\in\Co ([0,\tau],\R), \quad \varphi_{\overline{n}}(0)=0, \quad \varphi_{\overline{n}} \geq 0 \ \textnormal{ on } \ [0,\tau],\quad \varphi_{\overline{n}} \in \Co ^1((0,\tau],\R),$$
then some computations lead to
\begin{eqnarray*}
\hspace{-1.5cm}
\begin{array}{lll}
&\varphi_{\overline{n}}'(t)\\ 
=&\displaystyle \int_0^H G(\overline{n}(t,h)-n_H)\frac{\partial \overline{n}}{\partial t}(t,h)dh \\
=&\displaystyle \int_0^H G (\overline{n}(t,h)-n_H)\left(\frac{\partial^2 \overline{n}}{\partial h^2}(t,h)\right. +\left.L_h(p)(t,h)\left(\frac{n_H-\overline{n}(t,h)}{1+\chi(n(t,h)+n_H)}\right)\right)dh \\
=&-\displaystyle\int_0^H G'(\overline{n}(t,h)-n_H)\left|\frac{\partial \overline{n}}{\partial h}(t,h)\right|^2 dh \\
& + \ \displaystyle \int_0^H G(\overline{n}(t,h)-n_H) L_h(p)(t,h)\left(\frac{n_H-\overline{n}(t,h)}{1+\chi(n(t,h)+n_H)}\right) dh \\
\leq& 0
\end{array}
\end{eqnarray*}
since
$$G(\overline{n}(t,H)-n_H)=0, \qquad 1+\chi(n(t,h)+n_H)\geq 1/2, \qquad p(t,h)\geq 0,$$
for every $t\in(0,\tau]$ and a.e. $h\in[0,H]$. Thus we have
$$\overline{n}(t,h)\leq n_H, \quad \forall t\in[0,\tau], \quad \textnormal{a.e. } h\in[0,H].$$
Consequently (\ref{Eq:Bound_n}) holds. Now suppose that
$$(n_0,p_0,z_0)\in \X^\infty_{n_H}.$$
Since $D(\A)\cap \X^\infty_{n_H}$ is dense into $\X^\infty_{n_H}$, there exists a sequence $(n^k_0,p^k_0,z^k_0)_{k\geq 0}\in D(\A)\cap \X^\infty_{n_H}$ such that
$$\lim_{k\to \infty}\|(n_0,p_0,z_0)-(n^k_0,p^k_0,z^k_0)\|_{\X^\infty}=0.$$
For every $k\geq 0$, there exists a unique solution $(n^k,p^k,z^k)\in \Co([0,\tau],\X^\infty_{n_H})$ for the system (\ref{sistema3}) with initial condition $(n^k_0,p^k_0,z^k_0)$. Using (\ref{Eq:G}), for every $k\geq 0$, we get
\begin{eqnarray*}
\hspace{-2.5cm}
\begin{array}{lll}
&&\left(\begin{array}{cc}
n(t,\cdot) \\
p(t,\cdot) \\
z(t)
\end{array}\right)-\left(\begin{array}{cc}
n^k(t,\cdot) \\
p^k(t,\cdot) \\
z^k(t)
\end{array}\right)
\\
&=&T_{\A_m}(t)\left(\begin{array}{cc}
n_0-n^k_0 \\
p_0-p^k_0 \\
z_0-z^k_0
\end{array}\right)+\displaystyle \int_0^t T_{\A_m}(t-s)\left[f_m\left(\begin{array}{cc}
n(s,\cdot) \\
p(s,\cdot) \\
z(s)
\end{array}\right)-f_m\left(\begin{array}{cc}
n^k(s,\cdot) \\
p^k(s,\cdot) \\
z^k(s)
\end{array}\right)\right]ds
\end{array}
\end{eqnarray*}
for every $t\in [0,\tau]$, so
\begin{eqnarray*}
\hspace{-2.5cm}
\begin{array}{lll}
&&\left\|\left(\begin{array}{cc}
n(t,\cdot) \\
p(t,\cdot) \\
z(t)
\end{array}\right)-\left(\begin{array}{cc}
n^k(t,\cdot) \\
p^k(t,\cdot) \\
z^k(t)
\end{array}\right)\right\|_{\X^\infty} \vspace{0.1cm} \\
&\leq& \left\|\left(\begin{array}{cc}
n_0-n^k_0 \\
p_0-p^k_0 \\
z_0-z^k_0
\end{array}\right)\right\|_{\X^\infty}+\displaystyle \int_0^t (k_m+\lambda_m) \left\|\left(\begin{array}{cc}
n(s,\cdot) \\
p(s,\cdot) \\
z(s)
\end{array}\right)-\left(\begin{array}{cc}
n^k(s,\cdot) \\
p^k(s,\cdot) \\
z^k(s)
\end{array}\right)\right\|_{\X^\infty}ds \vspace{0.1cm} \\
&\leq& \left\|\left(\begin{array}{cc}
n_0-n^k_0 \\
p_0-p^k_0 \\
z_0-z^k_0
\end{array}\right)\right\|_{\X^\infty}
+\tau(k_m+\lambda_m) \max_{s\in[0,\tau]}
\left\|\left(\begin{array}{cc}
n(s,\cdot) \\
p(s,\cdot) \\
z(s)
\end{array}\right)-\left(\begin{array}{cc}
n^k(s,\cdot) \\
p^k(s,\cdot) \\
z^k(s)
\end{array}\right)\right\|_{\X^\infty}
\end{array}
\end{eqnarray*}
for every $t\in[0,\tau]$, since $((n,p,z),(n^k,p^k,z^k))\in \left(\X^\infty_{n_H+(2\chi)^{-1}}\cap B^m\right)^2$ and using (\ref{Eq:Sg_Bound}). Thus, we have
\begin{eqnarray*}
\max_{t\in[0,\tau]}
\left\|\left(\begin{array}{cc}
n(t,\cdot) \\
p(t,\cdot) \\
z(t)
\end{array}\right)-\left(\begin{array}{cc}
n^k(t,\cdot) \\
p^k(t,\cdot) \\
z^k(t)
\end{array}\right)\right\|_{\X^\infty}\\ \leq \left\|\left(\begin{array}{cc}
n_0-n^k_0 \\
p_0-p^k_0 \\
z_0-z^k_0
\end{array}\right)\right\|_{\X^\infty}
+\left(\frac{1}{2}\right)\max_{t\in[0,\tau]}
\left\|\left(\begin{array}{cc}
n(t,\cdot) \\
p(t,\cdot) \\
z(t)
\end{array}\right)-\left(\begin{array}{cc}
n^k(t,\cdot) \\
p^k(t,\cdot) \\
z^k(t)
\end{array}\right)\right\|_{\X^\infty}\end{eqnarray*}
whence
\begin{eqnarray*}
\max_{t\in[0,\tau]}
\left\|\left(\begin{array}{cc}
n(t,\cdot) \\
p(t,\cdot) \\
z(t)
\end{array}\right)-\left(\begin{array}{cc}
n^k(t,\cdot) \\
p^k(t,\cdot) \\
z^k(t)
\end{array}\right)\right\|_{\X^\infty}&\leq& 2\left\|\left(\begin{array}{cc}
n_0-n^k_0 \\
p_0-p^k_0 \\
z_0-z^k_0
\end{array}\right)\right\|_{\X^\infty}\to 0
\end{eqnarray*}
as $k$ goes to infinity, for every $t\in[0,\tau]$. Consequently (\ref{Eq:Bound_n}) holds and we have
$$(n,p,z)\in\Co\left([0,\tau],\X^\infty_{n_H}\cap B_m\right).$$
Some standard time extending properties of the solution allow to extend the solution $(n,p,z)$ over a maximal interval $[0,t_{\max})$.
\end{proof}

\subsection{Global existence and boundedness}

We now prove that the solution of (\ref{sistema3}) is global in time and that $n$ is bounded. We also give an example where $p$ and $z$ are bounded and go to extinction. We then deduce the result for (\ref{sistema1}).

\begin{theorem}
Suppose that operator $L_h$ has one of the shapes given in (\ref{Lh1}) or in (\ref{Lh2}). Then for every initial condition $(n_0,p_0,z_0)\in \X^\infty_{n_H}$, there exists a unique solution $(n,p,z)\in \Co \left([0,\infty),\X^\infty_{n_H}\right)$ for the system (\ref{sistema3}), that satisfies
\begin{eqnarray*}
n(t,h)\leq \max\{0,\sup_{h\in[0,H]}n_0(h)\}
\end{eqnarray*}
for every $t\geq 0$ and $h\in [0,H]$. Moreover, if
\begin{equation}\label{Hyp:m_p}
m_p>\frac{r}{\chi}
\end{equation}
holds true, then
\begin{equation*}
\lim_{t\to \infty}\|p(t,\cdot)\|_{L^\infty(0,H)}=0, \qquad \lim_{t\to \infty}z(t)=0.
\end{equation*}
\end{theorem}

\begin{proof}
Let $(n_0,p_0,z_0)\in \X^\infty_{n_H}$ and $(n,p,z)\in \Co \left([0,t_{\max}),\X^\infty_{n_H}\right)$ be
the solution of (\ref{sistema3}). Using the same argument of density as in the proof of Theorem \ref{Thm:Exist_Loc}, we only need to consider the case where the initial condition satisfies (\ref{Hyp:CI}). Because of the positivity of the solution, we have
$$\frac{\partial n}{\partial t}(t,h)\leq  \frac{\partial^2 n}{\partial h^2}(t,h).$$
We define the function
$$\varphi_n(t)=\int_0^H \kappa(n(t,h)-\KK^{n_0})dh.$$
We can show that
$$\varphi_n\in\Co ([0,t_{\max}),\R), \ \ \varphi_n(0)=0, \ \ \varphi_n \geq 0 \ \textnormal{ on } \ [0,t_{\max}), \ \ \varphi_n \in \Co ^1((0,t_{\max}),\R),$$
and 
\begin{eqnarray*}
\varphi_n'(t)&=&\displaystyle \int_0^H G(n(t,h)-\KK^{n_0})\frac{\partial n}{\partial t}(t,h)dh \\
&\leq& -\displaystyle \int_0^H G'(n(t,h)-\KK^{n_0})\left|\frac{\partial n}{\partial h}(t,h)\right|^2 dh\leq 0, \quad \forall t>0
\end{eqnarray*}
so
$$n(t,h)\leq \KK^{n_0}, \quad \forall t\geq 0, \quad \textnormal{a.e. } h\in[0,H].$$
To prove that the solution is global, suppose by contradiction that $t_{\max}<\infty$. Since $n$ is bounded, classical results (see \textit{e.g.} \cite{Pazy83}, Theorem 6.1.4, p. 185) imply that, either
$$\lim_{t\to t_{\max}} \|p(t,\cdot)\|_{L^\infty(0,H)}=\infty$$
or
$$\lim_{t\to t_{\max}}z(t)=\infty.$$
However, the former cannot hold since
\begin{equation}\label{Eq:Conv_p}
\frac{\partial p}{\partial t}(t,h)\leq \frac{\partial^2 p}{\partial h^2}(t,h)+\left(\frac{r}{\chi}-m_p\right)p(t,h), \ \forall t>0, \ \textnormal{a.e. } h\in[0,H]
\end{equation}
and the latter contradicts the fact that
\begin{equation}\label{Eq:Conv_z}
z'(t)\leq z(t)\left(\frac{k}{H}\int_0^H g(p)(t,h)dh-m\right), \quad \forall t>0.
\end{equation}
Consequently $t_{\max}=\infty$ and the solution is global in time. Suppose now that (\ref{Hyp:m_p}) holds and consider an initial consider that satisfies (\ref{Hyp:CI}). Since the solution is classical, we get the inequality (\ref{Eq:Conv_p}). An integration leads to
$$\frac{d}{dt}\int_0^H p(t,h)dh\leq \left(\frac{r}{\chi}-m_p\right)\int_0^H p(t,h)dh,$$
whence 
$$\lim_{t\to \infty}\int_0^H p(t,h)dh=0$$
by assumption (\ref{Hyp:m_p}) and 
$$\lim_{t\to \infty}z(t)=0$$
using (\ref{Eq:Conv_z}). Since $p(t,\cdot)\in H^2(0,H)\subset \Co^1([0,H])$ for every $t>0$, then
$$\lim_{t\to \infty}\|p(t,\cdot)\|_{L^\infty(0,H)}=0$$
which concludes the proof. 
\end{proof}

Using the change of variable (\ref{Eq:Change_Var}), we deduce the same result for the initial problem.

\begin{corollary}
Suppose that operator $L_h$ has one of the shapes given in (\ref{Lh1}) or in (\ref{Lh2}). Then for every initial condition $(n_0,p_0,z_0)\in \X^\infty_+$, there exists a unique solution $(n,p,z)\in \Co \left([0,\infty),\X^\infty_+\right)$ for the system (\ref{sistema1}), that satisfies
\begin{eqnarray*}
n(t,h)\leq \max\{n_H,\|n_0\|_{L^\infty}\}
\end{eqnarray*}
for every $t\geq 0$ and a.e. $h\in[0,H]$. Moreover, if (\ref{Hyp:m_p}) holds, then
\begin{eqnarray*}
\lim_{t\to \infty}\|p(t,\cdot)\|_{L^\infty(0,H)}=0, \qquad \lim_{t\to \infty}z(t)=0.
\end{eqnarray*} 
\end{corollary}

\section{Open questions and perspectives}

The well-posedness, positivity and asymptotic results that we proved in this article have wide range of applicability to reaction-diffusion model of plankton communities since the functional response $g$ covers several types of predation, such as Holling types I, II, III as well as Ivlev.  

The asymptotic results of extinction are obtained under a threshold condition related to phytoplankton population, stating that the mortality rate is bigger than the maximum growth rate. 

The case where this threshold condition (\ref{Hyp:m_p}) is not satisfied is an open question that will be investigated in a future work.

Another research direction concerns existence of steady states. The trivial equilibrium $(n_H \mathbf{1}_{[0,H]},0,0)\in \X$ clearly always exist. However the existence of non trivial steady states need deeper analysis. In \cite{AMV}, the authors proved numerically the existence of such non trivial equilibria for a slightly different model than the one presented in this paper.

Finally, re-cycling of the nutrient is contemplated in the boundary condition on function $n$ as a constant inflow of nutrient at position $H$. It could also be alternatively considered as a flux in the $n$-equation, but this would lead to different cases of studies in terms of modelling as well mathematical analysis.

\bigskip

\textbf{Acknowledgement}
This research was undertaken within the framework of the Epimath project, funded by Region Bourgogne Franche-Comt\'e.

\section*{References}

\bibliography{bib}

\end{document}